\documentclass{amsart}
\usepackage{graphicx}
\usepackage{amsmath}% amstex ?
\usepackage{amsthm}
\usepackage{amssymb,bbm}%
\usepackage[numbers, square]{natbib}
\usepackage{graphicx}
\usepackage{color}
\newcommand{\tc}{} %\textcolor{red}

\newcommand{\E}{\mathbb E}

\newcommand{\R}{\mathbb{R}}

\renewcommand{\P}{\mathbb{P}}

\theoremstyle{plain}
\newtheorem{theorem}{Theorem}[section]

\newtheorem{corollary}[theorem]{Corollary}

\theoremstyle{definition}

\theoremstyle{remark}

\begin{document}

\author{Friedrich G\"otze}
\address{Friedrich G\"otze, Faculty of Mathematics,
Bielefeld University,
P. O. Box 10 01 31,
33501 Bielefeld, Germany}
\email{goetze@math.uni-bielefeld.de}
\author{Dmitry Zaporozhets}
\address{Dmitry Zaporozhets\\
St.\ Petersburg Department of
Steklov Institute of Mathematics,
Fontanka~27,
 191011 St.\ Petersburg,
Russia}
\email{zap1979@gmail.com}

\title{Discriminant and root separation of integral polynomials}
\thanks{The work was done with the financial support of the Bielefeld University (Germany) in terms of project SFB 701.}

\begin{abstract}

Consider a random polynomial
$$
G_Q(x)=\xi_{Q,n}x^n+\xi_{Q,n-1}x^{n-1}+\dots+\xi_{Q,0}
$$
with independent coefficients uniformly distributed on $2Q+1$ integer points $\{-Q, \dots, Q\}$. Denote by $D(G_Q)$ the discriminant of $G_Q$. We show that there exists a constant $C_n$, depending on $n$ only such that for all $Q\ge 2$ the
distribution of $D(G_Q)$ can be approximated as follows
$$
\sup_{-\infty\leq a\leq b\leq\infty}\left|\mathbb{P}\left(a\leq \frac{D(G_Q)}{Q^{2n-2}}\leq b\right)-\int_a^b\varphi_n(x)\, dx\right|\leq\frac{C_n}{\log Q},
$$
where $\varphi_n$  denotes the distribution function of the discriminant of a random polynomial of degree $n$ with independent coefficients which are uniformly distributed on $[-1,1]$.

Let  $\Delta(G_Q)$ denote the minimal distance between the complex roots of $G_Q$. As an application we show that for any $\varepsilon>0$ there exists a constant $\delta_n>0$ such that  $\Delta(G_Q)$ is stochastically bounded from below/above  for all sufficiently large $Q$ in the following sense
$$
\mathbb{P}\left(\delta_n<\Delta(G_Q)<\frac1{\delta_n}\right)>1-\varepsilon .
$$

\end{abstract}

\maketitle

\section{Introduction}

Let
$$
p(x)=a_nx^n+a_{n-1}x^{n-1}+\dots+a_0=a_n(x-\alpha_1)\dots(x-\alpha_n)
$$
be a polynomial of degree $n$ with real or complex coefficients.

In this note we consider different asymptotic estimates when the degree $n$ is arbitrary but \emph{fixed}. Thus for non-negative functions $f,g$ we write $f\ll g$ if there exists a non-negative constant $C_n$ (depending on $n$ only) such that $f\leq C_ng$. We also write $f\asymp g$ if $f\ll g$ and $f\gg g$.

Denote by
$$
\Delta(p)=\min_{1\leq i<j\leq n}|\alpha_i-\alpha_j|
$$
the shortest distance between any two zeros of $p$.

In his seminal paper \tc{Mahler \cite{kM64}} proved that
\begin{equation}\label{331}
\Delta(p)\geq\sqrt{3}n^{-(n+2)/2}\frac{|D(p)|^{1/2}}{\left(|a_n|+\dots+|a_0|\right)^{n-1}},
\end{equation}
where
\begin{equation}\label{308}
D(p)=a_n^{2n-2}\prod_{1\leq i<j\leq n}(\alpha_i-\alpha_j)^2
\end{equation}
 \tc{denotes the} discriminant of $p(x)$. Alternatively, $D(p)$ \tc{is} given by the $(2n-1)$-dimensional determinant
\begin{multline}\label{932}
D(p)=(-1)^{n(n-1)/2}
\\\times
\left|\begin{matrix}
 & 1 & a_{n-1} &a_{n-2} & \ldots &  a_0 & 0 & \ldots &0&0 \\
 & 0 & a_n & a_{n-1} & \ldots &  a_1 & a_0 & \ldots &0&0 \\
 &\vdots&\vdots&\vdots&\ddots&\vdots&\vdots&\ddots&\vdots&\vdots\\
 &0 &0 &0 &\dots &a_{n-2} &a_{n-3}&\dots&a_1 &a_0 \\
 &n &(n-1)a_{n-1} &(n-2)a_{n-2} &\dots &0 &0 &\dots &0&0\\
 &0&na_n &(n-1)a_{n-1}  &\dots &a_1 &0 &\dots &0&0\\
&\vdots&\vdots&\vdots&\ddots&\vdots&\vdots&\ddots&\vdots&\vdots\\
 &0 &0 &0 &\dots &(n-1)a_{n-1} &(n-2)a_{n-2}&\dots&2a_2 &a_1 \\
\end{matrix}\right|.
\end{multline}

Define the height of the polynomial by $H(p)=\max_{0\leq i\leq n}|a_i|$. It follows immediately from \eqref{932} that
\begin{equation}\label{311}
|D(p)|\ll H(p)^{2n-2}.
\end{equation}

From now on we will always assume that the polynomial $p$ is integral (that is has integer coefficients).  Since the condition $D(p)\ne0$ implies $|D(p)|\geq1$  Mahler noted that \eqref{331} implies \begin{equation}\label{445}
\Delta(p)\gg H(p)^{-n+1},
 \end{equation}
provided that $p$ doesn't have multiple zeros.
The estimate \eqref{445} \tc{seems to be the best available lower bound  up  to now.} However, for $n\geq 3$ it is still not known how far \tc{it differs from the optimal lower bound.} Denote by $\kappa_n$ the infimum of $\kappa$ such that
$$
\Delta(p)> H(p)^{-\kappa}
$$
holds for all integral polynomials of degree $n$ without multiple zeros  and large enough height $H(p)$. It is easy to see that \eqref{445} is equivalent to $\kappa_n\leq n-1$. Also it is a simple exercise to show that $\kappa_2=1$ (see, e.g., \cite{BM10}). Evertse~\cite{jE04} showed that $\kappa_3=2$.

For $n\geq 4$ only estimates are known. At first, Mignotte~\cite{mM83} proved that $\kappa_n\geq n/4$ for $n\geq 2$. Later Bugeaud and Mignotte~\cite{BM04,BM10} have shown that $\kappa_n\geq n/2$ for even $n\geq 4$ and $\kappa_n\geq (n+2)/4$ for odd $n\geq 5$. Shortly after that Beresnivich, Bernik, and G\"otze~\cite{BBG10}, using completely different approach,  improved their result in the case of odd $n$: they obtained (as a corollary of more general counting result) that $\kappa_n\geq(n+1)/3$ for $n\geq2$. Recently Bugeaud and Dujella~\cite{BD14} achieved significant progress showing that $\kappa_n\geq(2n-1)/3$ for $n\geq4$ (see also~\cite{BD11} for irreducible polynomials).

Formulated in other terms the above results give answers to the question {\emph "How close to each other can two conjugate algebraic numbers of degree $n$ be?" } Recall that two complex algebraic numbers called conjugate (over $\mathbb Q$) if they are roots of the same irreducible integral polynomial (over $\mathbb Q$). Roughly speaking, if we consider a polynomial $p^*$ which minimizes $\Delta(p)$ among all integral polynomials of degree $n$ having the same height and without multiple zeros, then $\Delta(p^*)$
%has a polynomial decay rate
satisfies the following lower/upper bounds with respect to $H(p^*)$:
$$
H(p^*)^{-c_1n}\ll \Delta(p^*)\ll H(p^*)^{-c_2n},
$$
for some absolute constants $0<c_2\leq c_1$.
In this note, instead of considering the extreme polynomial $p^*$, we consider the behaviour of $\Delta(p)$ for a typical integral polynomial $p$. We prove that for ''most`` integral polynomials (see Section~\ref{1049} for a more precise formulation) we have
$$
\Delta(p)\asymp 1.
$$
We also show that the same estimate holds for ''most`` irreducible  integral polynomials (over $\mathbb Q$).

\bigskip

A related interesting problem is to study the distribution of discriminants of integral polynomials. To deal with it is convenient (albeit not necessary) to use probabilistic terminology. Consider some $Q\in\mathbb N$ and consider the class of all integral polynomials $p$ with $\deg(p)\leq n$ and $H(p)\leq Q$. The cardinality of this class is $(2Q+1)^{n+1}$. Consider the uniform probability measure  on this class so that the probability  of each polynomial is given by $(2Q+1)^{-n-1}$. In this sense, we may consider random polynomials
$$
G_Q(x)=\xi_{Q,n}x^n+\xi_{Q,n-1}x^{n-1}+\dots+\xi_{Q,0}
$$
with independent coefficients which are uniformly distributed on $2Q+1$ integer points $\{-Q, \dots, Q\}$. We are interested in the asymptotic behavior of $D(G_Q)$ when $n$ is fixed and $Q\to\infty$.

Bernik, G\"otze and Kukso~\cite{BGK08} showed that for $\nu\in[0,1/2]$
$$
\P(|D(G_Q|<Q^{2n-2-2\nu})\gg Q^{-2\nu}.
$$
Note that the case $\nu=0$ is consistent with \eqref{311}. It has been conjectured in~\cite{BGK08} that this estimate is optimal up to a constant:
\begin{equation}\label{211}
\P(|D(G_Q)|<Q^{2n-2-2\nu})\asymp Q^{-2\nu}.
\end{equation}
The conjecture turned out to be true for $n=2$: G\"otze, Kaliada, and Korolev~\cite{GKK13} showed that for $n=2$ and $\nu\in (0,3/4)$ it holds
$$
\P(|D(G_Q)|<Q^{2-2\nu})=2(\log 2+1)Q^{-2\nu}\left(1+O(Q^{-\nu}\log Q+Q^{2\nu-3/2}\log^{3/2}Q)\right).
$$
However, for $n=3$ and $\nu\in [0,3/5)$ Kaliada, G\"otze, and Kukso~\cite{KGK13} obtained the following asymptotic relation:
\begin{equation}\label{1202}
\P(|D(G_Q)|<Q^{4-2\nu})=\kappa Q^{-5\nu/3}\left(1+O(Q^{-\nu/3}\log Q+Q^{5\nu/3-1})\right),
\end{equation}
where the absolute constant $\kappa$ had been explicitly determined.
%was calculated in explicit way.

Recently Beresnevich, Bernik, and G\"otze~\cite{BBG15} extended the lower bound given by~\eqref{1202} to the full range of $\nu$ and to the arbitrary degrees n. They showed that for $0\leq\nu< n-1$ one has that
$$
\P(|D(G_Q|<Q^{2n-2-2\nu})\gg Q^{-n+3-(n+2)\nu/n}.
$$
They also obtained a similar result for resultants.

In this note we prove a limit  theorem for $D(G_Q)$. As a corollary, we obtain that ''with high probability`` (see Section~\ref{1049} for details) the following asymptotic equivalence holds:
$$
|D(P_Q)|\asymp Q^{2n-2}.
$$
The same estimate holds ''with high probability`` for irreducible polynomials.

For more comprehensive survey of the subject and a list of references, see~\cite{BBGK13}.

\section{Main results}\label{1049}

Let $\xi_0,\xi_1,\dots,\xi_{n}$ be independent random variables {\it uniformly} distributed on $[-1,1]$. Consider the random polynomial
$$
G(x)=\xi_nx^n+\xi_{n-1}x^{n-1}+\dots+\xi_1x+\xi_0
$$
and denote by $\varphi$ the distribution function of $D(G)$. It is easy to see that $\varphi$ has  compact support and $\sup_{x\in\R}\varphi(x)<\infty$.

\begin{theorem}\label{147}
Using the above notations we have
\begin{equation}\label{1133}
\sup_{-\infty\leq a\leq b\leq\infty}\left|\P\left(a\leq \frac{D(G_Q)}{Q^{2n-2}}\leq b\right)-\int_a^b\varphi(x)\, dx\right|\ll\frac1{\log Q}.
\end{equation}
\end{theorem}
How far is this estimate from being optimal? Relation \eqref{1202} shows that for $n=3$ the estimate $\log^{-1}Q$ can not be replaced by $Q^{-\varepsilon}$ for any $\varepsilon>0$. Otherwise it would imply that \eqref{211} holds for $\nu\leq\varepsilon/2$.

The proof of Theorem~\ref{147} will be given in Section~\ref{147}. Now let us derive some corollaries.

Relation~\eqref{311} means that $|D(G_Q)|\ll Q^{2n-2}$ holds a.s. It follows from Theorem~\ref{147} that with high probability the lower estimate holds as well.
\begin{corollary}\label{148}
For any $\varepsilon>0$ there exists $\delta>0$ (depending on $n$ only) such that for all sufficiently large $Q$
\begin{equation}\label{204}
\P(|D(G_Q)|>\delta Q^{2n-2})>1-\varepsilon.
\end{equation}
\end{corollary}
\begin{proof}
Since $\sup_{x\in\R}\varphi(x)<\infty$, it follows from \eqref{1133} that
$$
\P(|D(G_Q)|<\delta Q^{2n-2})\ll \delta+\frac{1}{\log Q},
$$
which completes the proof.
\end{proof}

As another corollary we obtain an estimate for $\Delta(G_Q)$.
\begin{corollary}\label{303}
For any $\varepsilon>0$ there exists $\delta>0$ (depending on $n$ only) such that for all sufficiently large $Q$
\begin{equation}\label{205}
\P(\delta<\Delta(G_Q)<\delta^{-1})>1-\varepsilon.
\end{equation}
\end{corollary}
\begin{proof}
For large enough $Q$ we have
$$
\P\left(|\xi_{Q,n}|>\frac{\varepsilon}{2}Q\right)>1-\varepsilon.
$$
Therefore it follows from \eqref{308} and \eqref{311} that with probability at least $1-\varepsilon$
$$
\Delta(G_Q)\leq\left(\frac2\varepsilon\right)^{2/n},
$$
which implies the upper estimate. The lower bound immediately follows from \eqref{204} and \eqref{331}.

\end{proof}

{\bf Remark on irreducibility.}
In order to consider $\Delta(G_Q)$ as distance between the  closest conjugate algebraic numbers of $G_Q$ we have to restrict ourselves to irreducible polynomials only. In other words  the distribution of the random polynomial $G_Q$ has to be conditioned on $G_Q$ being irreducible. It turns out that the  relations \eqref{204} and \eqref{205} with conditional versions of the left-hand sides still hold. This fact easily follows from the estimate
$$
\P(G_Q \text{is irreducible})\asymp1,
$$
which was obtained by van der Waerden~\cite{vW36}.

\section{Proof of Theorem~2.1}\label{1723}

For  $k\in\mathbb N$ the moments of $\xi_i$ and $\xi_{i,Q}$ are given by 
$$
\E\xi^{2k}_i=\frac1{2k+1},\quad \E\xi^{2k}_{i,Q}=\frac{2}{2Q+1}\sum_{j=1}^Qj^{2k}.
$$
Since
$$
\frac{Q^{2k+1}}{2k+1}=\int_0^Qt^{2k}dt\leqslant \sum_{j=1}^Qj^{2k}\leqslant\int_0^Q(t+1)^{2k}\,dt\leqslant\frac{(Q+1)^{2k+1}}{2k+1},
$$
we get
\begin{multline*}
\left|\frac{2}{2Q+1}\sum_{j=1}^Qj^{2k}-\frac{Q^{2k}}{2k+1}\right|=\frac{2}{2Q+1}\Big|\sum_{j=1}^Qj^{2k}-\frac{2Q+1}{2}\frac{Q^{2k}}{2k+1}\Big|
\\\leqslant\frac{2}{2Q+1}\Big|\sum_{j=1}^Qj^{2k}-\frac{Q^{2k+1}}{2k+1}\Big|+\frac{Q^{2k}}{2Q+1}
\\\leqslant\frac{2}{2Q+1}\cdot\frac{(Q+1)^{2k+1}-Q^{2k+1}}{2k+1}+\frac{Q^{2k}}{2Q+1}\leqslant2^{2k}Q^{2k-1},
\end{multline*}
which implies
\begin{equation}\label{133}
\left|\E\left(\frac{\xi_{i,Q}}Q\right)^{2k}-\E\xi^{2k}\right|\leqslant\frac{2^{2k}}{Q}.
\end{equation}
It follows from \eqref{932} that for all $k\in\mathbb N$
\begin{equation}\label{238}
\left|\E\,D^k\left(\frac{G_Q}Q\right)-\E\,D^k(G)\right|\leq
n^{nk}\sum_{k_0,\dots,k_n}\left|\prod_{i=0}^n\E\,\left(\frac{\xi_{i,Q}}Q\right)^{2k_i} -\prod_{i=0}^n\E\,\xi_{i}^{2k_i}\right|,
\end{equation}
where the summation is taken over at most $((2n-1)!)^k$ summands such that $k_0+\dots+k_n=k(n-1)$. Let us show that
\begin{equation}\label{123}
\left|\prod_{i=0}^n\E\,\left(\frac{\xi_{i,Q}}Q\right)^{2k_i} -\prod_{i=0}^n\E\,\xi_{i}^{2k_i}\right| \leq \frac{2^{2k_0+\dots+2k_n}}{Q}.
\end{equation}
We proceed by induction on $n$. The case $n=0$ follows from \eqref{133}. It holds
\begin{multline*}
\left|\prod_{i=0}^n\E\,\left(\frac{\xi_{i,Q}}Q\right)^{2k_i} -\prod_{i=0}^n\E\,\xi_{i}^{2k_i}\right|
\\\leq\left|\prod_{i=0}^{n-1}\E\,\left(\frac{\xi_{i,Q}}Q\right)^{2k_i} -\prod_{i=0}^{n-1}\E\,\xi_{i}^{2k_i}\right|\E\,\left(\frac{\xi_{n,Q}}Q\right)^{2k_n}
\\+\prod_{i=0}^{n-1}\E\,\xi_{i}^{2k_i} \left|\E\,\left(\frac{\xi_{n,Q}}Q\right)^{2k_n}-\E\,\xi_{0}^{2k_0}\right|.
\end{multline*}
Applying the induction assumption and \eqref{133}, we obtain \eqref{123}.

Thus, using \eqref{238}, \eqref{123}, and the relation $k_0+\dots+k_n=k(n-1)$ we get
\begin{equation}\label{430}
\left|\E\,D^k\left(\frac{G_Q}Q\right)-\E\,D^k(G)\right|\leq\frac{\gamma^k}{Q},
\end{equation}
where $\gamma$ depends on $n$ only.

Since $D(G)$ and $D(G_Q/Q)$ are bounded random variables, their characteristic functions
$$
f(t)=\E\,e^{iD(G)},\quad f_Q(t)=\E\,e^{iD(G_Q/Q)}
$$
are entire functions. Therefore  \eqref{430} implies that for all real $t$
\begin{equation}\label{743}
|f_Q(t)-f(t)|=\left|\sum_{k=1}^\infty i^k\frac{\E\,D^k(G_Q/Q)-\E\,D^k(G)}{k!}t^k\right|\leq\frac1Q\sum_{k=1}^\infty\frac{(\gamma |t|)^k}{k!}\leq\frac{\gamma |t|e^{\gamma |t|}}{Q}.
\end{equation}

Now we are ready to estimate the uniform distance between the distributions of $D(G)$ and $D(G_Q/Q)$ using the closeness of $f(t)$ and $f_Q(t)$. Let $F$ and $F_Q$ be distribution functions of $D(G)$ and $D(G_Q/Q)$. By Esseen's inequality, we get for any $T>0$
$$
\sup_{x}|F_Q(x)-F(x)|\leq\frac{2}{\pi}\int_{-T}^T\,\left|\frac{f_Q(t)-f(t)}{t}\right|\,dt+\frac{24}{\pi}\cdot\frac{\sup_{x\in\R}\varphi(x)}{T}.
$$
Applying \eqref{743}, we obtain that there exists a constant $C$ depending on $n$ only such that for any $T>0$
$$
 \sup_{-\infty\leq a\leq b\leq\infty}\left|\left(\P(a\leq D\left(\frac{G_Q}Q\right)\leq b\right)-\P(a\leq D(G)\leq b)\right|
 \leq C\left(\frac{Te^{\gamma T}}{Q}+\frac{1}{T}\right).
$$
%For that we need the following lemma taking from \cite{vP75}.
%\begin{lemma}\label{1042}
%Let $F(x)$ be a non-decreasing function, $G(x)$ a differentiable function of bounded variation on the real line, $f(t)$ and $g(t)$ the corresponding Fourier-Stieltjes transforms, and suppose that $F(-\infty)=G(-\infty), F(\infty)=G(\infty)$, $T$ is an arbitrary positive number. Suppose $\sup_x |G'(x)|\leq C$. Then for every $b>1/2\pi$ we have
%$$
%\sup_{x}|F(x)-G(x)|\leq b\int_{-T}^T\,\left|\frac{f(t)-g(t)}{t}\right|\,dt+r(b)\frac CT,
%$$
%where $r(b)$ is a positive constant depending only on $b$.
%\end{lemma}
Taking $T=\log Q/2\gamma$ completes the poof.

\section{Resultants}
Given polynomials
$$
p(x)=a_n(x-\alpha_1)\dots(x-\alpha_n),\quad q(x)=b_m(x-\beta_1)\dots(x-\beta_m),
$$
denote by $R(p,q)$ the resultant defined by
$$
R(p,q)=a_n^mb_m^n\prod_{i=1}^n\prod_{j=1}^m(\alpha_i-\beta_j).
$$
Obviously discriminants are essentially a specialization of  resultants via:
$$
D(p)=(-1)^{n(n-1)/2}a_n^{-1}R(p,p').
$$

Repeating the arguments from Section~\ref{1723} we obtain the following result. Consider the random polynomials
$$
G_Q(x)=\xi_{Q,n}x^n+\xi_{Q,n-1}x^{n-1}+\dots+\xi_{Q,1}x+\xi_{Q,0},
$$
$$
F_Q(x)=\eta_{Q,m}x^m+\eta_{Q,m-1}x^{m-1}+\dots+\eta_{Q,1}x+\eta_{Q,0}
$$
with independent coefficients uniformly distributed on $2Q+1$ points $\{-Q,\dots,Q\}$ and consider the random polynomials
$$
G(x)=\xi_nx^n+\xi_{n-1}x^{n-1}+\dots+\xi_1x+\xi_0,
$$
$$
F(x)=\eta_mx^m+\eta_{m-1}x^{m-1}+\dots+\eta_1x+\eta_0
$$
with independent coefficients uniformly distributed on $[-1,1]$. Denote by $\psi$ \tc{the distribution function} of $R(G,F)$. We have
\begin{equation*}
\sup_{-\infty\leq a\leq b\leq\infty}\left|\P\left(a\leq \frac{R(G_Q,F_Q)}{Q^{m+n}}\leq b\right)-\int_a^b\psi(x)\, dx\right|\ll\frac1{\log Q}.
\end{equation*}

\bigskip

\bigskip

{\bf Acknowledgments.} We are grateful to Victor Beresnevich, Vasili Bernik, and Zakhar Kabluchko for  \tc{useful} discussions. We also would like to thank Andrei Zaitsev for some remarks on notations.

\bibliographystyle{plain}
\bibliography{bib}

\end{document}